\documentclass[fontsize=11pt]{scrartcl}
\usepackage[latin1]{inputenc}
\usepackage{amsmath,amssymb,amsthm}
\usepackage{
           ,enumerate
           }
\usepackage[usenames,dvipsnames]{color}
\usepackage[left=2.7cm,top=3cm,right=2.8cm,bottom= 5cm]{geometry}
\newtheorem{theorem}{Theorem}[section]

\newtheorem{proposition}[theorem]{Proposition}
\newtheorem{corollary}[theorem]{Corollary}
\theoremstyle{definition}

\theoremstyle{remark}
\newtheorem{remark}[theorem]{Remark}
\numberwithin{equation}{section}

\newcommand{\be}{\begin{equation}}
\newcommand{\ee}{\end{equation}}
\newcommand{\bea}{\begin{eqnarray}}
\newcommand{\eea}{\end{eqnarray}}
\newcommand{\beas}{\begin{eqnarray*}}
\newcommand{\eeas}{\end{eqnarray*}}

\newcommand{\brak}[1]{\ensuremath{\left( #1 \right)}}
\newcommand{\crl}[1]{\ensuremath{ \left\{ #1 \right\} }}
\newcommand{\edg}[1]{\ensuremath{ \left[ #1 \right] }}
\newcommand{\ang}[1]{\ensuremath{ \left \langle #1 \right \rangle }}

\begin{document}

\title{Duality for increasing convex functionals with countably many marginal constraints}

\author{D.~Bartl\thanks{Department of Mathematics, University of Konstanz, 78464 Konstanz Germany.}
\and P.~Cheridito\thanks{ORFE, Princeton University, Princeton, NJ 08544, USA. 
Partially supported by National Science Foundation grant DMS-1515753.} 
\and M.~Kupper$^*$
\and L.~Tangpi\thanks{Faculty of Mathematics, University of Vienna, 1090, Vienna Austria. 
Partially supported by Vienna Science and Technology Fund (WWTF) grant MA 14-008.}}

\date{September 2016}

\maketitle

\begin{abstract}
In this work we derive a convex dual representation for increasing convex functionals 
on a space of real-valued Borel measurable functions defined on a countable product
of metric spaces. Our main assumption is that the functionals fulfill marginal 
constraints satisfying a certain tightness condition. In the special case where the marginal constraints 
are given by expectations or maxima of expectations, we obtain linear and sublinear versions of 
Kantorovich's transport duality and the recently discovered martingale transport duality on 
products of countably many metric spaces.\\[2mm]
\end{abstract}

\section{Introduction}

We consider an increasing convex functional $\phi : B_b \to \mathbb{R}$, where $B_b$ is the space of all bounded
Borel measurable functions $f : X \to \mathbb{R}$ defined on a countable product of metric spaces $X = \prod_n X_n$.
Under the assumption that there exist certain mappings $\phi_n$ defined on the 
bounded Borel measurable functions $g_n : X_n \to \mathbb{R}_+$, such that
$$
\phi(f)\le\sum_n\phi_n(g_n)\quad\mbox{whenever}
\quad f(x)\le\sum_n g_n(x_n)\mbox{ for all }x\in X,
$$
we show that $\phi$ can be represented as
\be \label{contrep}
\phi(f) = \max_{\mu \in ca^+} (\ang{f,\mu} - \phi^*_{C_b}(\mu)) \quad \mbox{for all } f \in C_b,
\ee
where $ca^+$ is the set of finite Borel measures, $C_b$ the set of bounded continuous functions
$f : X \to \mathbb{R}$, $\ang{f,\mu}$ the integral $\int f d\mu$, and
$\phi^*_{C_b}$ the convex conjugate defined by 
$$
\phi^*_{C_b}(\mu) := \sup_{f \in C_b} (\ang{f,\mu} - \phi(f)).
$$
We also provide equivalent conditions under which the representation \eqref{contrep}
extends to all bounded upper semicontinuous functions $f : X \to \mathbb{R}$.
In the special case, where the mappings $\phi_n$ are linear, our arguments 
can be generalized to cover functionals $\phi$ that are defined on spaces of unbounded 
functions $f : X \to \mathbb{R}$. This yields variants of the representation \eqref{contrep} for 
unbounded continuous and upper semicontinuous functions $f : X \to \mathbb{R}$.

As an application we derive versions of Kantorovich's transport duality and the recently 
discovered martingale transport duality in the case where the state space is a countable 
product of metric spaces. A standard Monge--Kantorovich transport problem consists in finding a 
probability measure on the product of two metric spaces with fixed marginals that minimizes the expectation of 
a given cost function. It is a linear optimization problem whose dual has the form of a 
subreplication problem (which, after changing the sign, becomes a superreplication problem).
Kantorovich first showed that there is no duality gap between the two problems under compactness 
and continuity assumptions in the seminal paper \cite{kanto}. Since then, 
the result has been generalized in various directions; see e.g. 
\cite{RR, Vil2, AG} for an overview. We establish linear and sublinear versions
of Kantorovich's duality for countable products of metric spaces and lower semicontinuous cost functions
(corresponding to upper semicontinuous functions $f : X \to \mathbb{R}$ in our setup).
It has been shown that in the case, where the state space is a finite product of Polish spaces, Kantorovich's
duality even holds for Borel measurable cost functions; see e.g. \cite{kellerer, bei-sch,bei-leo-sch}. However, 
we provide a counter-example illustrating that this is no longer true if the state space is 
a countable product of compact metric spaces. 

Martingale transport duality was discovered by
\cite{bei-hl-pen} and \cite{gal-hl-tou} in the context of model-independent finance by noting that 
the superreplication problem in the presence of liquid markets for European call and put options
can be viewed as the dual of a transport problem in which the optimization is carried out 
over the set of all martingale measures. While \cite{bei-hl-pen} considers a discrete-time 
model with finitely many marginal distributions, \cite{gal-hl-tou} studies a continuous-time model 
with just two marginal distributions. In this paper we obtain a martingale transport duality for 
countably many time periods and equally many marginal constraints (for martingale transport in 
continuous time, see e.g. \cite{DS2, GTT} and the references therein). Standard martingale transport duality 
describes a situation where a financial asset can be traded dynamically without transaction costs and 
any European derivative can efficiently be replicated with a static investment in European call and put options. 
From our general results, we obtain a sublinear generalization of the martingale transport duality corresponding to 
proportional transaction costs and incomplete markets of European call and put options. 
This extends the duality of \cite{DS1} to a setup with countably many time periods and 
markets for European options with all maturities. 

Our proofs differ from the standard arguments used in establishing Kantorovich duality and martingale transport duality 
in that they view the subreplication (or superreplication) problem as the primal problem and use the 
Daniell--Stone theorem to deduce that increasing convex functionals on certain function spaces have 
a max-representation with countably additive measures if they are continuous from above along point-wise 
decreasing sequences.

The rest of the paper is organized as follows: In Section \ref{sec2} we derive two general representation 
results for increasing convex functionals satisfying countably many tight marginal constraints.
In Section \ref{sec3} we focus on the special cases where the constraints are linear and sublinear.
In Section \ref{sec4} we derive linear and sublinear versions of Kantorovich's transport duality and the
martingale transport duality for countably many marginal constraints.

\section{Main representation results}\label{sec2}

Let $(X_n)$ be a countable (finite or countably infinite) family of metric spaces, and consider 
the product topology on $X = \prod_n X_n$. Denote by $C_b$, $U_b$ and $B_b$ all bounded functions 
$f: X \to \mathbb{R}$ that are continuous, upper semicontinuous or Borel measurable, respectively. 
Similarly, let $C_{b,n}$, $U_{b,n}$ and $B_{b,n}$ be all bounded functions $f :X_n \to \mathbb{R}$ that 
are continuous, upper semicontinuous or Borel measurable, respectively. By $ca^+$ we denote all finite 
Borel measures on $X$ and by $ca^+_n$ all finite Borel measures on $X_n$. For a measure $\mu \in ca^+$, 
we denote by $\mu_n$ the $n$-th marginal distribution, that is, 
$\mu_n:=\mu \circ \pi_n^{-1}$, where $\pi_n : X \to X_n$ is the projection on the $n$-th coordinate 
$x\mapsto \pi_n(x):=x_n$. For a sequence $g_n \in B^+_{b,n}$, where $B^+_{b,n}$ is the set of all 
bounded Borel measurable functions $f : X_n \to \mathbb{R}_+$, we define
$\oplus g := \sum_{n} g_n \circ \pi_n : X \to \mathbb{R}_+ \cup \crl{+\infty}$.
When we write $f_j \downarrow f$, we mean that $f_j$ is a decreasing sequence of functions that
converges point-wise to $f$.

Our goal in this section is to derive a dual representation for an increasing convex functional
$\phi : B_b \to \mathbb{R}$, where by increasing we mean that $\phi(f) \ge \phi(g)$ whenever $f \ge g$
and the second inequality is understood point-wise. For every $n$, let $\phi_n : B^+_{b,n} \to \mathbb{R}_+$ 
be a mapping satisfying the following tightness condition: for all $m,\varepsilon \in \mathbb{R}_+ \setminus \{0\}$,
there exists a compact set $K_n \subseteq X_n$ such that 
\be \label{margcond}
\phi_n(m 1_{K_n^c}) \le \varepsilon.
\ee
(In the special case where $\phi_n$ is given by $\phi_n(f) = \sup_{\nu \in \mathcal{P}_n} \int f d\nu$ for 
a set of Borel probability measures $\mathcal{P}_n$ on $X_n$, \eqref{margcond} means that
$\mathcal{P}_n$ is tight in the standard sense; see e.g. \cite{B}. A related condition for convex risk 
measures was introduced in \cite{FS}.)
We use the notation $\ang{f,\mu} := \int f d\mu$ and define the convex conjugate
$$
\phi^*_{C_b}: ca^+ \to \mathbb{R} \cup \crl{+\infty} \quad \mbox{by} \quad
\phi^*_{C_b}(\mu) := \sup_{f \in C_b} (\ang{f,\mu} - \phi(f)).
$$
Then the following holds:

\begin{theorem} \label{thm:Cb}
Let $\phi : B_b\to \mathbb{R}$ be an increasing convex functional satisfying $\phi(f) \le \sum_n \phi_n(g_n)$ 
for all $f \in B_b$ and $g_n \in B^+_{b,n}$ such that $f \le \oplus g$. Then
$$
\phi(f) = \max_{\mu \in ca^+} (\langle f,\mu \rangle - \phi^*_{C_b}(\mu))\quad \mbox{for all }f\in C_b.
$$
\end{theorem}

\begin{proof}
Fix $f \in C_b$ and let $(f_j)$ be a sequence in $C_b$ such that $f_j \downarrow 0$.
Since $\alpha \mapsto \phi(\alpha f)$ is a real-valued convex function on $\mathbb{R}$,
it is continuous. So, for a given constant $\varepsilon>0$, one can choose $\alpha \in (0,1)$ small enough 
such that $$(1-\alpha)\phi \brak{\frac{f}{1-\alpha}}-\phi(f) \le\varepsilon.
$$
By assumption, there exist compact sets $K_n\subseteq X_n$ such that 
$\sum_n \phi_n(g_n)\le \varepsilon$, where
$$g_n:=\frac{2}{\alpha}\|f_1\|_\infty 1_{K_n^c}.$$
By Tychonoff's theorem, $K:=\prod_n K_n\subseteq X$ is compact.
Since the function 
$$
\tilde\phi(\cdot):= \phi(\cdot + f)-\phi(f) : B_b \to \mathbb{R}
$$ is convex, one has
$$
\tilde\phi(f_j) \leq \frac{\tilde{\phi}(2 f_j 1_K)+\tilde{\phi} (2 f_1 1_{K^c})}{2}.
$$
By Dini's lemma, $f_j \to 0$ uniformly on the compact $K$. 
So, since $\lim_{\alpha \to 0} \tilde{\phi}(\alpha 1) = 0$, it follows by monotonicity that
$\tilde{\phi}(2 f_j 1_K) \to 0$. On the other hand, one obtains from 
$\frac{2}{\alpha} f_1 1_{K^c}\le \oplus g$ that
$$\phi \Big(\frac{2}{\alpha} f_1 1_{K^c}\Big)\le \sum_n \phi_n(g_n)\le\varepsilon,$$ and therefore,
$$
\tilde\phi(2f_1 1_{K^c})\le\alpha \phi\Big(\frac{2}{\alpha}f_1 1_{K^c}\Big)
+(1-\alpha)\phi\Big(\frac{f}{1-\alpha}\Big)-\phi(f)\le 2\varepsilon.
$$
This shows $\phi(f+f_j) \downarrow \phi(f)$. By the Hahn--Banach extension theorem, 
there exists a positive linear functional $\psi : C_b \to \mathbb{R}$ such that 
$$
\psi(g)\le \tilde{\phi}(g) = \phi(f+g) - \phi(f)\quad\mbox{for all } g \in C_b.
$$
Since $\psi(g_j) \downarrow 0$ for every sequence $(g_j)$ in $C_b$ satisfying $g_j \downarrow 0$,
one obtains from the Daniell--Stone theorem (see e.g., Theorem 4.5.2 in \cite{D}) that there exists a $\nu \in ca^+$ such that 
$\psi(g) = \ang{g,\nu}$ for all $g \in C_b$. It follows that $\phi(f) + \phi^\ast_{C_b}(\nu) \le \ang{f,\nu}$,
which together with $\phi(f) \ge \sup_{\mu \in ca^+} (\ang{f,\mu} - \phi^*_{C_b}(\mu))$ yields
$$
\phi(f) = \max_{\mu\in ca^+} (\langle f,\mu\rangle - \phi^\ast_{C_b}(\mu)).
$$
\end{proof}

The next result gives conditions under which the dual representation of Theorem \ref{thm:Cb} extends to 
the set of bounded upper semicontinuous functions $U_b$. We call a subset $\Lambda$ of $ca^+$ 
sequentially compact if every sequence in $\Lambda$ has a subsequence that converges to some 
$\mu \in \Lambda$ with respect to the topology $\sigma(ca^+,C_b)$.

\begin{theorem} \label{thm:usc}
Let $\phi : B_b \to \mathbb{R}$ be an increasing convex functional satisfying the assumption of 
Theorem \ref{thm:Cb}. Then the lower level sets 
\[\Lambda_a := \{\mu \in ca^+ : \phi^*_{C_b}(\mu) \le a\}, \quad a \in \mathbb{R},\] 
are sequentially compact, and the following are equivalent:
\begin{itemize}
\item[{\rm (i)}] $\phi(f) = \max_{\mu \in ca^+} (\langle f,\mu\rangle  - \phi^*_{C_b}(\mu))$ for all $f \in U_b$
\item[{\rm (ii)}] $\phi(f_j) \downarrow \phi(f)$ for all $f \in U_b$ and every  sequence $(f_j)$ in $C_b$ satisfying $f_j \downarrow f$
\item[{\rm (iii)}] $\phi(f) = \inf_{g \in C_b, \, g \ge f} \phi(g)$ for all $f \in U_b$
\item[{\rm (iv)}] $\phi^*_{C_b}(\mu) = \phi^*_{U_b}(\mu):=\sup_{f\in U_b}\left(\ang{f,\mu} -\phi(f)\right)$ for all $\mu \in ca^+$.
\end{itemize}
\end{theorem}

\begin{proof} 
It is clear that for all $a \in \mathbb{R}$, $\Lambda_a$ is $\sigma(ca^+,C_b)$-closed.
Moreover, for all $\mu \in ca^+$,
$$
\phi^*_{C_b}(\mu) \ge \sup_{x \in \mathbb{R}_+} (\ang{x1,\mu} - \phi(x1))
= \gamma(\ang{1,\mu}),$$
where $\gamma : \mathbb{R}_+ \to \mathbb{R} \cup \{+ \infty\}$ is the increasing convex function 
given by 
$$
\gamma(y) := \sup_{x \in \mathbb{R}_+} (xy - \phi(x1)).
$$
Since $\phi$ is real-valued, $\gamma$ has the property $\lim_{y \to + \infty} \gamma(y)/y = +\infty$, 
from which it follows that the right-continuous inverse $\gamma^{-1} : \mathbb{R} \to \mathbb{R}_+$ given by
$$\gamma^{-1}(x) := \sup \{y \in \mathbb{R}_+ : \gamma(y) \le x\} \quad \mbox{with} \quad\sup \emptyset := 0,
$$
is increasing and satisfies $\lim_{x \to +\infty} \gamma^{-1}(x)/x = 0$.
For every $\varepsilon>0$ there exist $m\in\mathbb{N}$ such that $(a+1)/m\le\varepsilon$
and compact sets $K_n\subseteq X_n$ so that $\sum_n \phi_n(m 1_{K^c_n})\le 1$.
Since $m 1_{K^c}\le\oplus g$ for the compact $K:= \prod_n K_n$ and $g_n:=m 1_{K^c_n}$, one has  
$\phi(m1_{K^c})\le\sum_n\phi_n(m 1_{K_n^c})\le 1$. Moreover, the 
product topology on $X$ is metrizable and $m 1_{K^c}$ is lower semicontinuous.
Therefore, there exists a sequence $(g_j)$ in $C_b$ such that $g_j \uparrow m 1_{K^c}$. Since 
$\phi(g_j) \le \phi(m 1_{K^c}) \le 1$, one has for all $\mu \in \Lambda_a$,
$$
m\mu(K^c) = \sup_j \ang{g_j,\mu} \le \sup_j (\ang{g_j,\mu} -\phi(g_j)+1) \le \phi^*_{C_b}(\mu)+1 \le a +1.
$$
In particular, $\mu(K^c) \le \varepsilon$ and $\mu(X)= \ang{1,\mu} \le\gamma^{-1}\big(\phi^\ast_{C_b}(\mu)\big)\le\gamma^{-1}(a)$. 
Now one obtains from the first half of Prokhorov's theorem 
(see e.g. Theorem 5.1 in \cite{B}) that $\Lambda_a$ is sequentially compact.

(i) $\Rightarrow $ (ii): 
Fix $f \in U_b$ and assume $(f_j)$ is a sequence in $C_b$ such that $f_j \downarrow f$. If (i) holds, 
there exists a sequence $(\mu_j)$ in $ca^+$ such that
$$
\phi(f_j)= \ang{f_j,\mu_j} - \phi^*_{C_b}(\mu_j) \le \|f_1\|_\infty \ang{1, \mu_j} - \phi^*_{C_b}(\mu_j)
\le \|f_1\|_\infty \gamma^{-1}(\phi^*_{C_b}(\mu_j)) - \phi^*_{C_b}(\mu_j).
$$
It follows that $(\mu_j)$ is in $\Lambda_a$ for some $a \in \mathbb{R}$ large enough. 
Therefore, after possibly passing to a subsequence, $\mu_j$ converges to a measure $\mu \in \Lambda_a$
in $\sigma(ca^+, C_b)$. Clearly, $\phi^*_{C_b}$ is $\sigma(ca^+, C_b)$-lower semicontinuous, and so
$$
\phi^*_{C_b}(\mu) \le \liminf_j \phi^*_{C_b}(\mu_j).
$$
Moreover, for every $\varepsilon > 0$, there is a $k$ such that 
$\ang{ f_k, \mu} \le \ang{ f,\mu} + \varepsilon$.
Now choose $j \ge k$ such that $\ang{f_k, \mu_j} \le \ang{f_k, \mu} + \varepsilon$.
Then 
$$
\ang{f_j, \mu_j} \le \ang{f_k,\mu_j} \le \ang{f_k, \mu}  + \varepsilon \le \ang{f,\mu} + 2 \varepsilon.
$$
It follows that $\limsup_j \ang{f_j,\mu_j} \le \ang{f,\mu}$, and therefore,
$$
\lim_j \phi(f_j) = \lim_j (\ang{f_j,\mu_j}  - \phi^*_{C_b}(\mu_j)) 
\le \ang{f,\mu} - \phi^*_{C_b}(\mu) \le \phi(f),
$$
showing that $\phi(f_j) \downarrow \phi(f)$.

(ii) $\Rightarrow$ (iii) follows from the fact that for every 
$f \in U_b$, there exists a sequence $(f_j)$ in $C_b$ such that $f_j \downarrow f$. 

(iii) $\Rightarrow$ (iv): 
It is immediate from the definitions that $\phi^*_{U_b} \ge \phi^*_{C_b}$.
On the other hand, if (iii) holds, then for every $f \in U_b$, there is a sequence $(f_j)$ in $C_b$ such that 
$f_j \ge f$ and $\phi(f_j) \downarrow \phi(f)$. In particular, 
$$
\sup_j (\ang{f_j, \mu} - \phi(f_j)) \ge \ang{f,\mu} - \phi(f),
$$
from which one obtains $\phi^*_{C_b} \ge \phi^*_{U_b}$.

(iv) $\Rightarrow$ (i):
Fix $f \in U_b$. It is a direct consequence of the definition of $\phi^*_{U_b}$ that
$$
\phi(f) \ge \sup_{\mu \in ca^+} (\langle f,\mu\rangle - \phi_{U_b}^*(\mu)) = 
\sup_{\mu \in ca^+} (\langle f,\mu\rangle - \phi_{C_b}^*(\mu)).
$$
On the other hand, there exists a sequence $(f_j)$ in $C_b$ such that $f_j \downarrow f$.
Since 
$$
\ang{f_j,\mu} \le \ang{f_1,\mu} \le \|f_1\|_{\infty} \ang{1,\mu} \le \|f_1\|_{\infty}
\gamma^{-1}\big(\phi^\ast_{C_b}(\mu)\big),
$$ it follows from Theorem 
\ref{thm:Cb} that one can choose $a \in \mathbb{R}$ large enough such that 
$$
\phi(f_j) = \ang{f_j, \mu_j} - \phi^*_{C_b}(\mu_j)
$$
for a sequence $(\mu_j)$ in the sequentially compact set $\Lambda_a$. After passing to a 
subsequence, $\mu_j$ converges to a $\mu$ in $\sigma(ca^+, C_b)$. 
Then it follows as above that
$$
\phi(f) \le \lim_j \phi(f_j) = \lim_j (\ang{f_j, \mu_j}- \phi^*_{C_b}(\mu_j)) 
\le \ang{f, \mu}  - \phi^*_{C_b}(\mu), 
$$
from which one obtains $\phi(f) = \max_{\mu \in ca^+} (\langle f,\mu\rangle - \phi^*_{C_b}(\mu))$.
\end{proof}

\section{Linear and sublinear marginal constraints} 
\label{sec3}

In this section we assume the $X_n$ to be Polish spaces and the
mappings $\phi_n : B^+_{b,n} \to \mathbb{R}$ to be of the form 
$$
\phi_n(g) = \sup_{\nu_n \in \mathcal{P}_n} \ang{g,\nu_n},
$$
where $\mathcal{P}_n$ is a non-empty convex $\sigma(ca^+_n, C_{b,n})$-compact set of Borel probability measures
on $X_n$. Then all $\phi_n$ are increasing and sublinear.
Moreover, they have the translation property 
$$
\phi_n(g+m) = \phi_n(g) + m, \quad g \in B_{b,n}, \; m \in \mathbb{R},
$$ and it follows from Prokhorov's theorem that they satisfy the tightness condition \eqref{margcond};
see e.g. \cite{B}. By $\mathcal{P}$ we denote the set of Borel probability measures $\mu$ on 
the product $X = \prod_n X_n$ whose marginal distributions $\mu_n:=\mu\circ\pi_n^{-1}$ 
are in $\mathcal{P}_n$ for all $n$. Under these circumstances the following holds:

\begin{proposition} \label{prop:sub}
Let $\phi : B_b \to \mathbb{R}$ be an increasing convex functional satisfying
\begin{equation}\label{eq:dom}
\phi(f) \le m + \sum_{n} \phi_n(g_n)
\end{equation}
whenever $f \le m + \oplus g$ for some $m \in \mathbb{R}$ and $g_n \in B^+_{b,n}$. Then 
\begin{equation}\label{rep1}
\phi(f)=\max_{\mu\in \mathcal{P}}(\langle f,\mu \rangle-\phi^\ast_{C_b}(\mu))\quad\mbox{for all } f\in C_b.
\end{equation}
If in addition, $\phi^\ast_{C_b}(\mu)=\phi^\ast_{U_b}(\mu)$ for all $\mu \in \mathcal{P}$, the 
representation \eqref{rep1} extends to all $f\in U_b$.
\end{proposition}

\begin{proof}
One obtains from Theorem \ref{thm:Cb} that 
$$
\phi(f) = \max_{\mu \in ca^+} (\ang{f,\mu} - \phi^*_{C_b}(\mu)) \quad \mbox{for all } f \in C_b,
$$
and from Theorem \ref{thm:usc} that the representation holds for all $f \in U_b$ if 
$\phi^*_{C_b} = \phi^*_{U_b}$. So the proposition follows if we can show 
that $\phi^*_{C_b}(\mu) = + \infty$ for all $\mu \in ca^+ \setminus \mathcal{P}$. 
To do that we fix a $\mu \in ca^+ \setminus \mathcal{P}$. If it is not a probability measure, then
$$
\phi_{C_b}^*(\mu) \ge \sup_{m \in\mathbb{R}} (\ang{m,\mu} - \phi(m)) \ge \sup_{m \in\mathbb{R}} (\ang{m,\mu} - m) = + \infty.
$$
On the other hand, if $\mu$ is a probability measure, but does not belong to $\mathcal{P}$, one
obtains from the Hahn--Banach separation theorem 
that there exist $n$ and $g_n \in C_{b,n}$ such that $\ang{g_n ,\mu_n} > \phi_n(g_n)$.
Moreover, since $\phi_n$ has the translation property, $g_n$ can be shifted until it is non-negative.
Then
$$
\phi(m g_n \circ \pi_n) \le \phi_n(m g_n) = m \phi_n(g_n) \quad \mbox{for all } m \in \mathbb{R}_+,
$$
and therefore,
$$
\phi^*_{C_b}(\mu) \ge \sup_{m \in \mathbb{R}_+} (\ang{m g_n \circ \pi_n, \mu} - \phi(m g_n \circ \pi_n))
\ge \sup_{m \in \mathbb{R}_+} m(\ang{g_n, \mu_n} - \phi_n(g_n)) = + \infty.
$$
\end{proof}

In the next step we concentrate on the special case where every $\mathcal{P}_n$ 
consists of just one Borel probability measure $\nu_n$ on $X_n$. 
Then the mappings $\phi_n$ are of the form $\phi_n(g) = \ang{g,\nu_n}$. In particular, they are linear, and 
the representation \eqref{rep1} can be extended to unbounded functions $f$.

Let us denote by $\mathcal{P}(\nu)$ the set of all Borel probabilities on $X$ with marginals 
$\mu_n = \nu_n$. Furthermore, let $B$ be the space of all Borel measurable functions 
$f : X \to \mathbb{R}$, $U$ the subset of 
upper semicontinuous functions $f : X \to \mathbb{R}$ and $B^+_n$ the set of all 
Borel measurable functions $f : X_n \to \mathbb{R}_+$. Consider the following sets:
\beas
G(\nu) &:= &\Big\{\oplus g : (g_n) \in \prod_{n} B^+_n \mbox{ such that } 
\sum_{n} \langle g_n,\nu_n \rangle < +\infty \Big\}\\
B(\nu) &:=& \{f \in B: |f| \leq \oplus g \mbox{ for some } \oplus g \in G(\nu)\}\\
U(\nu) &:=& \{f\in U: f^+\in B_b \mbox{ and } f^- \in B(\nu) \}.
\eeas
Note that $G(\nu)$ is not contained in $B(\nu)$ since a function $\oplus g \in G(\nu)$ 
can take on the value $+\infty$. But one has $\ang{\oplus g, \mu} = \sum_n \ang{g_n, \nu_n} < + \infty$
for all $\oplus g \in G(\nu)$ and $\mu \in \mathcal{P}(\nu)$. So $G(\nu)$ is contained in $L^1(\mu)$, and
every $\oplus g \in G(\nu)$ is finite $\mu$-almost surely.

\begin{proposition} \label{prop:linear}
Let $\phi : B(\nu) \to \mathbb{R}$ be increasing and convex such that
\begin{equation}\label{eq:dom2}
\phi(f) \le m + \sum_{n} \langle g_n,\nu_n \rangle
\end{equation}
if $f \le m+\oplus g$ for some $m \in \mathbb{R}$ and $\oplus g\in G(\nu)$.
Moreover, assume that 
$$\phi^\ast_{C_b}(\mu)=\phi^\ast_{U(\nu)}(\mu):=\sup_{f\in U(\nu)}\left(\langle f,\mu \rangle-\phi(f)\right)
\quad \mbox{for all } \mu\in \mathcal{P}(\nu).
$$ 
Then
$$
\phi(f)=\max_{\mu \in \mathcal{P}(\nu)}(\langle f,\mu \rangle-\phi^\ast_{C_b}(\mu))\quad
\mbox{for all } f \in B(\nu) \cap (U(\nu) + G(\nu)).
$$
\end{proposition}

\begin{proof} 
By Proposition \ref{prop:sub}, one has 
$$
\phi(f) = \max_{\mathcal{P}(\nu)} (\ang{f,\mu} - \phi^*_{C_b}(\mu)) \quad \mbox{for all  } f \in C_b.
$$
Furthermore, for given $f \in U(\nu)$, there exists a sequence $(f_j)$ in $C_b$ such that 
$f_j \downarrow f$, and it follows as in the proof of (iv) $\Rightarrow$ (i) in Theorem \ref{thm:usc} 
that there exists a $\mu \in \mathcal{P}(\nu)$ such that $\phi(f) \le \ang{f,\mu} -\phi^*_{C_b}(\mu)$.
Since on the other hand, 
$$
\phi(f) \ge \sup_{\mu \in \mathcal{P}(\nu)} (\ang{f,\mu} -\phi^*_{U(\nu)}(\mu)) 
= \sup_{\mu \in \mathcal{P}(\nu)} (\ang{f,\mu} -\phi^*_{C_b}(\mu)),
$$ one obtains
$$
\phi(f)= \max_{\mu\in \mathcal{P}(\nu)}(\langle f,\mu\rangle-\phi^\ast_{C_b}(\mu)).
$$
Next, notice that it follows, as in the proof of Theorem \ref{thm:Cb}, from the Hahn--Banach 
extension theorem that 
$$
\phi(f) = \max_{\psi \in B'(\nu)}
\left(\psi(f) - \phi^\ast(\psi)\right) \quad \mbox{for all } f \in B(\nu),
$$
where $B'(\nu)$ is the algebraic dual of $B(\nu)$ and 
$\phi^\ast(\psi) := \sup_{f\in B(\nu)}\left(\psi(f)-\phi(f) \right)$,
$\psi \in B'(\nu)$. For $\psi \in B'(\nu)$ with $\phi^*(\psi) < + \infty$, one has for all $\oplus g \in G(\nu) \cap B(\nu)$,
$$
\psi(\oplus g) - \sum_n \ang{g_n, \nu_n} \le 
\psi(\oplus g) - \phi(\oplus g) \le \phi^*(\psi) < + \infty,
$$
and therefore, $\psi(\oplus g) \le \sum_n \ang{g_n, \nu_n}$. On the other hand,
if one sets $g^N_n := g_n \wedge N$ for $n \le N$ and $g^N_n := 0$ for $n > N$, then 
$$
\psi(N^2 - \oplus g^N) \le N^2 - \sum_{n=1}^N \ang{g_n \wedge N, \nu_n},
$$
from which one obtains
$$
\psi(\oplus g) \ge \lim_N \psi(\oplus g^N) \ge \lim_N \sum_{n=1}^N \ang{g_n \wedge N, \nu_n}
= \sum_n \ang{g_n, \nu_n}.
$$
This shows that $\psi(\oplus g) = \sum_n \ang{g_n, \nu_n}$ for all $\oplus g \in G(\nu) \cap B(\nu)$,
and as a result, 
\beas
&& \phi(f - \oplus g) = \max_{\psi \in B'(\nu)} (\psi(f - \oplus g) - \phi^*(\psi))
= \phi(f) - \sum_n \ang{g_n,\nu_n}
\eeas
for all $f \in B(\nu)$ and $\oplus g \in G(\nu)$.
Finally, let $f\in B(\nu)$ be of the form $f= \oplus g + h$ for $\oplus g\in G(\nu)$ and $h \in U(\nu)$.
Then $f - \oplus g \in U(\nu)$ and $\oplus g\in G(\nu) \cap B(\nu)$. So
\beas
&& \phi(f)-\sum_{n}\langle g_n, \nu_n \rangle =\phi(f-\oplus g)
=\max_{\mu \in \mathcal{P}(\nu)}(\langle  f-\oplus g,\mu\rangle-\phi^\ast_{C_b}(\mu))\\
&&=\max_{\mu \in \mathcal{P}(\nu)}(\langle f,\mu\rangle-\phi^\ast_{C_b}(\mu))-\sum_{n}\langle g_n, \nu_n\rangle,
\eeas
and hence, $\phi(f) = \max_{\mu \in \mathcal{P}(\nu)}(\langle  f,\mu\rangle-\phi^\ast_{C_b}(\mu))$.
\end{proof}

\setcounter{equation}{0}
\section{Generalized (martingale) transport dualities}
\label{sec4}

In this section we derive generalizations of Kantorovich's transport duality and the more recently 
introduced martingale transport duality. 

\subsection{Generalized transport dualities}

As in Section \ref{sec3}, let $X_n$ be Polish spaces. We first study the case 
where a probability measure $\nu_n$ is given on each $X_n$. For given 
$f \in B(\nu)$, consider the minimization problem 
\be \label{minproblem}
\phi(f) := \inf\Big\{m+\sum_{n} \ang{g_n,\nu_n} : 
m \in\mathbb{R}, \, \oplus g \in G(\nu) \mbox{ such that } m+\oplus g\geq f \Big\}.
\ee

\begin{remark} \label{rem:diffdual}
Up to a different sign, \eqref{minproblem} can be viewed as a generalized version of the dual of a transport problem.
A standard transport problem in the sense of Kantorovich consists in finding a Borel probability 
measure $\mu$ on the product of two metric spaces $X_1 \times X_2$ with given marginals $\nu_1$ 
and $\nu_2$ that minimizes the expectation $\mathbb{E}^{\mu} c$ of a 
cost function $c : X_1 \times X_2 \to \mathbb{R}$. The (negative of the) dual problem is a minimization 
problem of the form 
\be \label{minfinite}
\inf \sum_{n=1}^2 \ang{g_n,\nu_n},
\ee
where the infimum is taken over all $g_n \in L^1(\nu_n)$ such that $\oplus g \ge f := -c$. To relate
\eqref{minproblem} to \eqref{minfinite} more closely, note that $\oplus g^1 - \oplus g^2$ is well-defined 
for all $\oplus g^1 \in G(\nu)$ and $\oplus g^2 \in G(\nu) \cap B(\nu)$. So instead of \eqref{minproblem},
we could have defined $\phi(f)$ equivalently as
$$\inf \bigg\{ \sum_{n} \ang{g^1_n - g^2_n,\nu_n} : 
\begin{array}{l} 
\oplus g^1 \in G(\nu), \, \oplus g^2 \in G(\nu) \cap B(\nu)\\
\text{such that }\oplus g^1 - \oplus g^2 \geq f
\end{array}\bigg\}.$$
Indeed, it is clear that the above infimum minorizes $\phi(f)$. On  the other hand, since
$$\lim_{N \to + \infty} \sum_{n=1}^N \ang{g^2_n \wedge N, \nu_n} = \sum_n \ang{g^2_n, \nu_n},$$
it cannot be strictly smaller.
\end{remark}

As a consequence of the results in Section \ref{sec3}, one obtains the following version of 
Kantorovich's transport duality with countably many marginal distributions:

\begin{corollary} \label{cor:lintransport}
\be \label{cor1}
\phi(f) = \max_{\mu \in \mathcal{P}(\nu)} \ang{f,\mu} \quad \mbox{for all } f \in B(\nu) \cap (U(\nu) + G(\nu)).
\ee
\end{corollary}

\begin{proof}
Clearly, $\phi(f) < + \infty$ for all $f \in B(\nu)$. On the other hand, since $\mathcal{P}(\nu)$ is non-empty
(it contains the product measure $\otimes_n \nu_n$), one has
$$
m+ \sum_n \ang{g_n,\nu_n} \ge \sup_{\mu \in \mathcal{P}(\nu)} \ang{f,\mu} > - \infty
$$
for all $m \in \mathbb{R}$, $\oplus g \in G(\nu)$ and $f \in B(\nu)$ such that $m + \oplus g \ge f$.
It follows that $\phi : B(\nu) \to \mathbb{R}$ is an increasing sublinear functional satisfying
$$
\phi(f) \ge \sup_{\mu \in \mathcal{P}(\nu)} \ang{f,\mu} \quad \mbox{for all } f \in B(\nu).
$$
In particular, $\phi(0) = 0$, and
$\phi^*_{C_b}(\mu) = \phi^*_{U(\nu)}(\mu) = 0$ for all $\mu \in \mathcal{P}(\nu)$.
So the duality \eqref{cor1} follows from Proposition \ref{prop:linear}.
\end{proof}

\begin{remark} \label{rem:trans}
If $X$ is a finite product of Polish spaces, it can be shown that 
$$
\phi(f) = \sup_{\mu \in \mathcal{P}(\nu)} \ang{f,\mu} \quad \mbox{for all } f \in B_b; 
$$
see e.g. \cite{kellerer, bei-sch,bei-leo-sch}. But for countably infinite products, there may arise 
a duality gap; that is, it may happen that 
$$
\phi(f) > \sup_{\mu \in \mathcal{P}(\nu)} \ang{f,\mu} \quad \mbox{for some } f \in B_b.
$$
For instance, if $X$ is the product of $X_n = \{0,1\}$, $n \in \mathbb{N}$, and
$\nu_n=\frac{1}{2}(\delta_0+\delta_1)$ for all $n$, then $f := \liminf_n \pi_n$
belongs to $B_b$, and it follows from Fatou's lemma that 
$$
\ang{f,\mu} \le \liminf_n \ang{\pi_n,\mu} =\frac{1}{2} \quad \mbox{for all }
\mu \in \mathcal{P}(\nu).
$$
On the other hand, assume $f \le m + \oplus g$ for some $m \in \mathbb{R}$ and $\oplus g\in G(\nu)$.
Since $$\frac{1}{2} \sum_{n}(g_n(0)+g_n(1))=\sum_{n}\langle g_n,\nu_n\rangle<+\infty,$$ one has 
$\sum_{n}g_n(x_n)<+\infty$, for all $x \in X$, and therefore,
$$
\inf_{k \in\mathbb{N}} \min_{(y_1,\dots,y_k)\in\{0,1\}^k}
\Big(\sum_{n \le k} g_n(y_n)+ \sum_{n > k} g_n(x_n)\Big)
= \sum_n \min_{y_n\in\{0,1\}}g_n(y_n) \le \sum_n \ang{g_n,\nu_n}.
$$
Consequently,
$$
1= \inf_{k \in \mathbb{N}} \min_{(y_1,\dots,y_k)\in\{0,1\}^k} f(y_1,\dots,y_k,1,1,\dots) \le
m+\sum_n \ang{g_n,\nu_n},
$$
from which it follows that $\phi(f)\ge 1$.
\end{remark}

In the more general case, where the $\phi_n : B_{b,n} \to \mathbb{R}$ are
sublinear functionals given by $$\phi_n(g) = \sup_{\nu_n \in \mathcal{P}_n} \ang{g,\nu_n}$$ for 
non-empty convex $\sigma(ca_n^+, C_{b,n})$-compact sets of Borel probability measures $\mathcal{P}_n$ on $X_n$,
we obtain a generalized Kantorovich duality with countably many sets of marginal distributions.
As in Section \ref{sec3}, $\mathcal{P}$ denotes the set of probability distributions such that 
$\mu_n \in \mathcal{P}_n$ for all $n$. Compared to Corollary \ref{cor:lintransport}, one has to 
modify the definition of $\phi$ slightly:
\be \label{min}
\phi(f) := \inf\Big\{m+\sum_{n} \phi_n(g_n) : 
m \in\mathbb{R}, \, g_n \in B^+_{b,n} \mbox{ such that } m+\oplus g\geq f \Big\}.
\ee
Then, an application of Proposition \ref{prop:sub} and essentially the same arguments as in the 
proof of Corollary \ref{cor:lintransport} yield the following duality:

\begin{corollary} \label{cor:subtransport}
\be \label{cor2}
\phi(f) = \max_{\mu \in \mathcal{P}} \ang{f,\mu} \quad \mbox{for all } f \in U_b.
\ee
\end{corollary}

\begin{proof}
As in the proof of Corollary \ref{cor:lintransport} it is
easy to see that $\phi : B_b \to \mathbb{R}$ is an increasing sublinear functional such that
$$
\phi(f) \ge \sup_{\mu \in \mathcal{P}} \ang{f,\mu} \quad \mbox{for all } f \in B_b.
$$
Since $\mathcal{P}$ is non-empty (it contains all product measures $\otimes_n \nu_n$ for $\nu_n \in \mathcal{P}_n$),
it follows that $\phi(0) = 0$ and $\phi^*_{C_b}(\mu) = \phi^*_{U_b}(\mu) = 0$ for all $\mu \in \mathcal{P}$.
So \eqref{cor2} follows from Proposition \ref{prop:sub}.
\end{proof}

\subsection{Generalized martingale transport dualities}
\label{sec:martingal}

Next, we derive linear and sublinear versions of the martingale transport duality with 
countably many marginal  constraints. Let $X_n$ be non-empty closed subsets of 
$\mathbb{R}^d$ and model the discounted prices of $d$ financial assets by 
$S_0 := s_0 \in \mathbb{R}^d$ and $S_n(x) := x_n$, $x \in X = \prod_n X_n$. 
The corresponding filtration is given by $\mathcal{F}_n:=\sigma(S_j : j \leq n)$.

We first assume that each space $X_n$ carries a single Borel probability measure $\nu_n$.
Moreover, we suppose that money can be lent and borrowed at the same interest rate and 
European options with general discounted payoffs $g_n \in B^+_n$ can be bought at 
price $\ang{g_n,\nu_n}$ (we suppose they either exist as structured products or they 
can be synthesized by investing in more standard options; see e.g. \cite{BL} for the form of $\nu_n$
if European call options exist with maturity $n$ and all strikes). A function $\oplus g \in G(\nu)$ 
then corresponds to a static option portfolio costing $\sum_n \ang{g_n,\nu_n}$. In addition, 
the underlying can be traded dynamically. The set $\mathcal{H}$ of dynamic trading strategies 
consists of all finite sequences $h_1, \dots, h_N$ such that each $h_n$ is an $\mathbb{R}^d$-valued
$\mathcal{F}_{n-1}$-measurable function on $X$. An $h \in \mathcal{H}$ generates gains of the form 
$$
(h\cdot S)_N :=\sum_{n=1}^N h_n \cdot (S_n - S_{n-1}).
$$
A triple $(m, \oplus g , h) \in \Theta := \mathbb{R} \times G(\nu) \times \mathcal{H}$ describes a 
semi-static trading strategy with cost $m + \sum_{n} \ang{ g_n,\nu_n}$ and outcome 
$m + \oplus g + (h \cdot S)_N$.

A strategy $(m, \oplus g , h) \in \Theta$ is said to be a model-independent arbitrage if 
$$
m+\sum_n \ang{g_n,\nu_n} \leq 0 \quad \mbox{and} \quad m + \oplus g+ (h \cdot S)_N >0.
$$
Similarly, we call a strategy $(m, \oplus g , h) \in \Theta$ a uniform arbitrage if
$$
m+\sum_n \ang{g_n,\nu_n} < 0 \quad \mbox{and} \quad m + \oplus g+ (h \cdot S)_N \ge 0.
$$
Consider the superhedging functional
\be \label{superhedging}
\phi(f) := \inf \bigg\{m + \sum_{n} \ang{ g_n,\nu_n}:
\begin{array}{l} 
(m,\oplus g, h) \in \Theta \mbox{ such that }\\
m + \oplus g+ (h \cdot S)_N \geq f
\end{array}\bigg\}, \ee
and denote by ${\mathcal M}(\nu)$ the set of probability measures $\mu \in \mathcal{P}(\nu)$ under which
$S$ is a $d$-dimensional martingale.

\begin{remark}
The static part of a semi-static strategy in $\Theta$ consists of a cash position and 
a portfolio of options with non-negative payoffs. But one could extend the set of strategies 
to include portfolios with outcomes $\oplus g^1 - \oplus g^2 + (h \cdot S)_N$
and prices $\sum_n \ang{g^1_n - g^2_n,\nu_n}$ for $g^1 \in G(\nu)$, $g^2 \in G(\nu) \cap B(\nu)$ and $h \in \mathcal{H}$. 
It follows as in Remark \ref{rem:diffdual} that this would not change the superhedging functional \eqref{superhedging},
the definition of a model-independent arbitrage or the definition of a uniform arbitrage.
\end{remark}

The following corollary extends the superhedging duality of \cite{bei-hl-pen} to a model with
countably many time periods and contains a model-independent fundamental theorem of asset pricing 
as a consequence. For $x \in X_n \subseteq \mathbb{R}^d$, denote by $|x|$ the Euclidean norm of $x$.

\begin{corollary} \label{cor:marttransport}
Assume that $\int_{X_n} \lvert x\rvert \,d\nu_n(x)<+\infty$ for all $n$.
Then the following are equivalent:
\begin{itemize}
\item[{\rm (i)}] 
there is no model-independent arbitrage,
\item[{\rm (ii)}] 
there is no uniform arbitrage,
\item[{\rm (iii)}]
$\mathcal{M}(\nu)\neq\emptyset$.
\end{itemize}
Moreover, if {\rm (i)--(iii)} hold, then 
\be
\label{rep3}
\phi(f)=\max_{\mu\in {\mathcal M}(\nu)} \ang{f,\mu} \quad \mbox{for all } f\in B(\nu) \cap (U(\nu) + G(\nu)).
\ee
\end{corollary}

\begin{proof}
It is clear that (i) implies (ii) since for every uniform arbitrage $(m, \oplus g , h)$, there exists 
an $\varepsilon >0$ such that $(m + \varepsilon, \oplus g , h)$ is a model-independent arbitrage.

Furthermore, if (iii) holds, there exists a $\mu$ in ${\mathcal M}(\nu)$. Let
$(m, \oplus g, h) \in \Theta$ be a strategy such that $m+ \oplus g + (h \cdot S)_N > 0$. Then
$\mathbb{E}^{\mu} (h\cdot S)_N^- \le m^+ + \sum_n \ang{g_n,\nu_n} < + \infty$, and it follows 
that $(h \cdot S)_n$, $n = 1, \dots, N$, is a martingale under $\mu$ (see e.g. \cite{JS}). In particular, 
$\mathbb{E}^{\mu} (h\cdot S)_N = 0$, and therefore,
$$
m + \sum_n \ang{g_n,\nu_n} = \ang{m+\oplus g + (h\cdot S)_N,\mu} > 0.
$$
So there is no model-independent arbitrage, showing that (i) is satisfied.

Now, let us assume (ii). Then 
$\phi : B(\nu)\to\mathbb{R}\cup\{-\infty\}$ is an increasing sublinear functional with the property
that $\phi(f) \leq m+\sum_{n \ge 1} \ang{g_n,\nu_n}$
whenever $f \le m + \oplus g$ for some $m \in \mathbb{R}$ and $\oplus g\in G(\nu)$.
If there is no uniform arbitrage, one has $\phi(0) = 0$, from which it follows by 
subadditivity that $\phi(f) > - \infty$ for all $f \in B(\nu)$. Moreover, if
$$
m + \oplus g + (h \cdot S)_N \ge f
$$
for $(m, \oplus g, h) \in \Theta$ and $f \in B(\nu)$, one has for all $\mu \in \mathcal{M}(\nu)$,
$$
\mathbb{E}^{\mu} (h \cdot S)^-_N \le m^+ + \sum_n \ang{g_n, \nu_n} + \ang{f^-, \mu} < + \infty.
$$
It follows as above that $\mathbb{E}^{\mu} (h \cdot S)_N=0$, and therefore, $m + \sum_n \ang{g_n,\nu_n} \ge \ang{f,\mu}$.
This implies $\phi(f) \ge \ang{f,\mu}$, and consequently,
$\phi^*_{C_b}(\mu) =\phi^*_{U(\nu)}(\mu) = 0$ for all $\mu \in \mathcal{M}(\nu)$. So if we can show that 
\begin{equation}\label{eq10}
\phi^\ast_{C_b}(\mu) = +\infty\quad\mbox{for all }\mu\in \mathcal{P}(\nu) \setminus \mathcal{M}(\nu),
\end{equation}
we obtain from Proposition \ref{prop:linear} that \eqref{rep3} holds, which in turn, implies that
$\mathcal{M}(\nu)$ cannot be empty. 

To show \eqref{eq10}, let $\mu \in \mathcal{P}(\nu)$. If $\mathbb{E}^{\mu} S_1 = s_0$ and 
$\mathbb{E}^{\mu} \edg{v(x_1, \dots , x_n)\cdot(x_{n+1}-x_n)} = 0$ for all $n \ge 1$ and 
every bounded continuous function $v : \prod_{j=1}^n X_j \to \mathbb{R}^d$, then 
$S$ is a martingale under $\mu$, and therefore, $\mu \in \mathcal{M}(\nu)$. So for 
$\mu \in \mathcal{P}(\nu) \setminus \mathcal{M}(\nu)$, there must exist a continuous function $f \in B(\nu)$ 
with $\ang{f,\mu} > 0$ such that $f$ is either of the form 
$f(x) = v \cdot (x_1-s_0)$ for a vector $v \in \mathbb{R}^d$ or
$f(x) = v(x_1, \dots, x_n) \cdot (x_{n+1}-x_n)$ for some $n \ge 1$ and a bounded continuous function 
$v: \prod_{j=1}^n X_j \to \mathbb{R}^d$. For $k \in \mathbb{N}$, $f^k := f \wedge k$ is bounded above
and $f^k_k := f^k \vee (-k)$ bounded. By monotonicity, one has $\phi(f^k) \le \phi(f) \le 0$. Moreover,
$$
f^k_k(x) = f^k(x) + (k+f(x))^- \le f^k(x) + w^k(x),
$$
where 
$$
w^k(x) := (c|x_n| - k/2)^+ + (c|x_{n+1}| - k/2)^+
$$
and $c \in \mathbb{R}_+$ is a bound on $|v|$. Since $w^k$ is in $G(\nu)$, 
one gets
$$
\phi(w^k) \le \int_{X_n} (c|x_n| - k/2)^+ d\nu_n(x_n) 
+ \int_{X_{n+1}} (c|x_{n+1}| - k/2)^+ d\nu_{n+1}(x_{n+1})
\to 0 $$
for $k \to + \infty$.
So for $k$ large enough, one obtains from monotonicity and subadditivity that
$$
\ang{f^k_k,\mu} - \phi(f^k_k) \ge \ang{f^k,\mu} - \phi(f^k) - \phi(w^k) \ge 
\ang{f^k,\mu} - \phi(w^k) > 0,
$$
and as a result, $\phi^*_{C_b}(\mu) = + \infty$. 
\end{proof}

Now, we extend the setting of Corollary \ref{cor:marttransport} by adding friction and incompleteness.
To simplify the presentation we assume that each $X_n$ is a non-empty closed subset of $\mathbb{R}^d_+$.
As above, $S_0 = s_0 \in \mathbb{R}^+_d$, $S_n(x) = x_n$, $x \in X$, and  
the set of dynamic trading strategies $\mathcal{H}$ is given by all finite sequences 
$h_1, \dots, h_N$ of $\mathcal{F}_{n-1}$-measurable mappings $h_n : X \to \mathbb{R}^d$.
But now we assume that dynamic trading incurs proportional transaction costs. 
If the bid and ask prices of asset $i$ are given by $(1- \varepsilon_i) S^i_n$ and $(1+\varepsilon_i) S^i_n$
for a constant $\varepsilon_i \ge 0$, a strategy $h \in \mathcal{H}$ leads to an outcome of 
$$
h(S) := \sum_{n=1}^N \sum_{i=1}^d h^i_n (S^i_n - S^i_{n-1}) - 
\varepsilon_i |h^i_n - h^i_{n-1}| S^i_{n-1}, \quad \mbox{where } h^i_0 := 0.
$$
(We assume there are no initial asset holdings. So there is a transaction cost at time $0$.
On the other hand, asset holdings at time $N$ are valued at $h_N \cdot S_N$ and do not have to 
be converted into cash.) Similarly, a European option with payoff $g_n \in B^+_n$ at time $n$ is assumed to cost
$$
\phi_n(g_n) = \sup_{\nu_n \in \mathcal{P}_n} \ang{g_n,\nu_n},
$$
where $\mathcal{P}_n$ is a non-empty convex $\sigma(ca^+_n,C_{b,n})$-compact set of Borel probability 
measures on $X_n$ (non-linear prices $\phi_n(g_n)$ may arise if e.g. not enough liquidly traded vanilla 
options exist to exactly replicate the payoffs $g_n$, or there are positive bid-ask spreads in the vanilla options market;
see e.g. \cite{DS1}). Compared to the frictionless case, we now have to require a little bit more integrability of the option 
portfolio. As in Section \ref{sec3}, we denote by $\mathcal{P}$ the set of all Borel probability measures 
$\mu$ on $X = \prod_n X_n$ with marginal distributions in $\mathcal{P}_n$. We introduce the sets
\beas
G(\mathcal{P}) &:= &\Big\{\oplus g : (g_n) \in \prod_{n} B^+_n \mbox{ such that } 
\sum_{n} \phi_n(g_n) < +\infty \Big\},\\
B(\mathcal{P}) &:=& \{f \in B: |f| \leq \oplus g \mbox{ for some } \oplus g \in G(\mathcal{P})\},
\eeas
and consider option portfolios with payoffs $\oplus g$ for functions $g_n \in B^+_n$ such that
$\sum_n \phi_n(g_n) < + \infty$. We still denote the set of corresponding strategies $(m,\oplus g, h)$ by $\Theta$. 
The corresponding superhedging functional is given by 
\be \label{superhedgingbounded}
\phi(f):=\bigg\{m + \sum_{n} \phi_n(g_n): 
\begin{array}{l} 
(m,\oplus g, h) \in \Theta \mbox{ such that}\\
m + \oplus g+ h(S) \geq f
\end{array}\bigg\}.
\ee

A model-independent arbitrage now consists of a strategy $(m,\oplus g,h) \in \Theta$ such that 
$$
m + \sum_n \phi_n(g_n) \le 0 \quad \mbox{and} \quad m + \oplus g + h(S) > 0,
$$
and a uniform arbitrage of a strategy $(m,\oplus g,h) \in \Theta$ satisfying 
$$
m + \sum_n \phi_n(g_n) < 0 \quad \mbox{and} \quad m + \oplus g + h(S) \ge 0.
$$
The set of martingale measures has to be extended to the set 
$\mathcal{M}(\mathcal{P})$ of all measures $\mu \in \mathcal{P}$ satisfying
\be \label{cps}
(1- \varepsilon_i)S^i_n \le \mathbb{E}^{\mu}[S^i_N \mid \mathcal{F}_n] \le (1+\varepsilon_i) S^i_n
\quad \mbox{for all } i,N \mbox{ and } n\leq N.
\ee

The following is a variant of Corollary \ref{cor:marttransport} with friction and incompleteness. 
It extends the duality result of \cite{DS1} to the case of countably many time periods and European 
options with all maturities.

\begin{corollary}
\label{cor:gen:transport}
Assume that $\lim_{k\to+\infty}\sup_{\nu_n\in\mathcal{P}_n}\int_{X_n} (\lvert x\rvert-k)^+ \,d\nu_n(x)=0$ for all $n$.
Then the following are equivalent:
\begin{itemize}
\item[{\rm (i)}] there is no model-independent arbitrage,
\item[{\rm (ii)}] there is no uniform arbitrage,
\item[{\rm (iii)}]
$\mathcal{M}(\mathcal{P}) \neq \emptyset$.
\end{itemize}
Moreover, if {\rm (i)--(iii)} hold, then 
\be
\label{rep4}
\phi(f)=\max_{\mu\in\mathcal{M}(\mathcal{P})} \ang{f,\mu} \quad \mbox{for all } f\in U_b.
\ee
\end{corollary}

\begin{proof}
As in the proof of Corollary \ref{cor:marttransport}, the implication 
(i) $\Rightarrow$ (ii) is straight-forward since the existence of a uniform arbitrage 
implies the existence of a model-independent arbitrage. 

If (iii) holds, there exists a $\mu$ in $\mathcal{M}(\mathcal{P})$. So if $(m, \oplus g, h) \in \Theta$ 
is a strategy with $m + \oplus g + h(S) > 0$, then
$$
\mathbb{E}^{\mu} h(S)^- \le m^+ + \ang{\oplus g, \mu} \le m^+ + \sum_n \phi_n(g_n) < + \infty.
$$
Moreover, for all $i$,
\beas
&& \sum_{n=1}^N h^i_n (S^i_n  -S^i_{n-1}) -\varepsilon_i |h^i_n - h^i_{n-1}| S^i_{n-1}\\
&=& \sum_{n=1}^N \sum_{k=1}^n (h^i_k - h^i_{k-1}) (S^i_n  -S^i_{n-1}) -\varepsilon_i |h^i_n - h^i_{n-1}| S^i_{n-1}\\
&=& \sum_{k=1}^N (h^i_k - h^i_{k-1}) (S^i_N - S^i_{k-1}) -\varepsilon_i |h^i_k - h^i_{k-1}| S^i_{k-1}.
\eeas
Denote $\tilde{S}^i_n = \mathbb{E}^{\mu}[S^i_N \mid \mathcal{F}_n]$ and
$$
Y_n = \sum_{k=1}^n \sum_{i=1}^d (h^i_k - h^i_{k-1}) (\tilde{S}^i_n - S^i_{k-1}) -\varepsilon_i |h^i_k - h^i_{k-1}| S^i_{k-1}
\quad \mbox{with } Y_0 = 0.
$$
Then $Y_N = h(S)$, and if the conditional expectation is understood in the general sense of \cite{JS}, one has
\beas
 &&\mathbb{E}^{\mu}[Y_n \mid \mathcal{F}_{n-1}] - Y_{n-1} \\
 &&=\sum_{i=1}^d 
 \mathbb{E}^{\mu} \edg{(h^i_n - h^i_{n-1}) (\tilde{S}^i_n - S^i_{n-1}) -\varepsilon_i |h^i_n - h^i_{n-1}| S^i_{n-1} \mid 
\mathcal{F}_{n-1}}\\
&&= \sum_{i=1}^d (h^i_n - h^i_{n-1}) (\tilde{S}^i_{n-1}- S^i_{n-1}) -\varepsilon_i |h^i_n - h^i_{n-1}| S^i_{n-1} \le 0.
\eeas
So $Y_n$ is of the form $Y_n = M_n -A_n$, where $M_n$ is a generalized $\mu$-martingale starting at $0$
and $$A_n = \sum_{k=1}^n Y_{k-1} -\mathbb{E}^{\mu}[Y_k \mid \mathcal{F}_{k-1}]$$ a predictable increasing
process. Since $\mathbb{E}^{\mu} M^-_N \le \mathbb{E}^{\mu} Y^-_N = \mathbb{E}^{\mu} h(S)^- < + \infty$,
one obtains from \cite{JS} that $(M_n)$ is a true $\mu$-martingale. In particular, 
$h(S) = M_N - A_N$ is $\mu$-integrable with $\mathbb{E}^{\mu} h(S) \le 0$. Therefore,
$$
m + \sum_n \phi_n(g_n) \ge m + \sum_n \ang{g_n, \nu_n} \ge \mathbb{E}^{\mu} \edg{m + \oplus g + h(S)} > 0,
$$
which shows that $(m,\oplus g,h)$ cannot be a model-independent arbitrage.

Finally, let us assume (ii). Then it follows as in the proof of Corollary \ref{cor:lintransport} that
$\phi$ is a real-valued increasing convex functional on $B(\mathcal{P})$ such that $\phi(0) = 0$
and $\phi(f) \leq m+\sum_n \phi_n(g_n)$ whenever $f \le m + \oplus g$ for some $m \in \mathbb{R}$ 
and $\oplus g \in G(\mathcal{P})$.
Moreover, if
$$
m + \oplus g + h(S) \ge f
$$
for a strategy $(m, \oplus g, h) \in \Theta$ and $f \in B(\mathcal{P})$, one has for all $\mu \in \mathcal{M}(\mathcal{P})$,
$$
\mathbb{E}^{\mu} h(S)^- \le m^+ + \sum_n \ang{g_n, \nu_n} + \ang{f^-, \mu}
\leq m^+ + \sum_n \phi_n(g_n) + \ang{f^-,\mu} < + \infty.
$$
So it follows as above that $\mathbb{E}^{\mu} h(S) \le 0$, and therefore, $m + \sum_n \phi_n(g_n) \ge \ang{f,\mu}$.
This implies that $\phi(f) \ge \ang{f,\mu}$, and consequently,
$\phi^*_{C_b}(\mu) =\phi^*_{U_b}(\mu) = 0$ for all $\mu \in \mathcal{M}(\mathcal{P})$. It remains to show that
\begin{equation}\label{eq11}
\phi^\ast_{C_b}(\mu) = +\infty\quad\mbox{for }\mu\in \mathcal{P} \setminus \mathcal{M}(\mathcal{P}).
\end{equation}
Then Proposition \ref{prop:sub} implies \eqref{rep4}, and thereby also (iii).

To show \eqref{eq11}, fix $\mu \in \mathcal{P}$. If 
$$
(1-\varepsilon_i) s^i_0 \le \mathbb{E}^{\mu} x^i_N \le (1+\varepsilon_i) s^i_0
$$
as well as
$$
\mathbb{E}^{\mu}[v(x_1, \dots, x_n) (x^i_N - (1+ \varepsilon_i) x^i_n)] \le 0 
\quad \mbox{and} \quad
\mathbb{E}^{\mu}[v(x_1, \dots, x_n) ((1-\varepsilon_i) x^i_n - x^i_N)] \le 0,
$$
for all $i$, $N$, $n\leq N$ and every bounded continuous function $v : \prod_{j=1}^n X_j \to \mathbb{R}_+$, then 
$$
(1-\varepsilon_i) S^i_n \le \mathbb{E}^{\mu}[S^i_N \mid \mathcal{F}_n] \le (1+\varepsilon_i) S^i_n \quad 
\mbox{for all } i,N \mbox{ and } n\leq N.
$$
So for $\mu \in \mathcal{P} \setminus \mathcal{M}(\mathcal{P})$, there exists an $f$ 
with $\ang{f,\mu} > 0$, where $f$ is of the form $f(x) = x^i_N - (1+\varepsilon_i) s^i_0$,
$f(x) = (1-\varepsilon_i) s^i_0 - x^i_N$, $f(x) = v(x_1, \dots, x_n) (x^i_N - (1+ \varepsilon_i) x^i_n)$
or $f(x) = v(x_1, \dots, x_n) ((1- \varepsilon_i) x^i_n - x^i_N)$ for a bounded continuous function 
$v: \prod_{j=1}^n X_j \to \mathbb{R}_+$. For $k \in \mathbb{N}$, define $f^k := f \wedge k$ and
$f^k_k := f^k \vee (-k)$. By monotonicity, one has $\phi(f^k) \le \phi(f) \le 0$. Moreover,
$$
f^k_k(x) = f^k(x) + (k+f(x))^- \le f^k(x) + (c|x^i_n| - k/2)^+ + (c|x^i_N| - k/2)^+,
$$
for $c \in \mathbb{R}_+$ large enough. Since $w^k(x) := (c|x^i_n| - k/2)^+ + (c|x^i_N| - k/2)^+$ belongs to
$G(\mathcal{P})$ one gets
$$
\phi(w^k) \le \phi_n( (c|x_n| - k/2)^+ ) + \phi_N((c|x_N| - k/2)^+ ) \to 0 \quad \mbox{for } k \to +\infty
$$
by our assumption on $\mathcal{P}_n$. So for $k$ large enough, one has
$$
\ang{f^k_k,\mu} - \phi(f^k_k) \ge \ang{f^k,\mu} - \phi(f^k) - \phi(w^k) \ge 
\ang{f^k,\mu} - \phi(w^k) > 0,
$$
and therefore, $\phi^*_{C_b}(\mu) = + \infty$. 
\end{proof}

\bibliographystyle{amsplain}

\end{document}